\documentclass[11pt, a4paper]{amsart}
\usepackage{amssymb, amsmath, amsfonts, mathrsfs, amsthm}
\usepackage[utf8]{inputenc}
\usepackage{lmodern}
\usepackage[T1]{fontenc}
\usepackage[a4paper,includeheadfoot,margin=2.54cm]{geometry}
\usepackage[matrix, arrow, curve]{xy}

\DeclareMathOperator{\Exc}{\mathrm{Exc}}

\DeclareMathOperator{\PP}{\mathbb{P}}

\DeclareMathOperator{\dl}{\mathrm{dl}}
\DeclareMathOperator{\vdl}{\mathrm{vdl}}

\DeclareMathOperator{\Bir}{\mathrm{Bir}}
\DeclareMathOperator{\Aut}{\mathrm{Aut}}
\DeclareMathOperator{\Psaut}{\mathrm{Psaut}}

\numberwithin{equation}{section}

\begin{document} 

\title{A remark on polycyclic groups of birational automorphisms}
\date{}
\author{Aleksei Golota} 

\newtheorem{theorem}{Theorem}[section] 
\newtheorem*{ttheorem}{Theorem}
\newtheorem{lemma}[theorem]{Lemma}
\newtheorem{proposition}[theorem]{Proposition}
\newtheorem*{conjecture}{Conjecture}
\newtheorem{corollary}[theorem]{Corollary}

{\theoremstyle{remark}
\newtheorem{remark}[theorem]{Remark}
\newtheorem{example}[theorem]{Example}
\newtheorem{notation}[theorem]{Notation}
\newtheorem{question}[theorem]{Question}
}

\theoremstyle{definition}
\newtheorem{construction}[theorem]{Construction}
\newtheorem{definition}[theorem]{Definition}

\begin{abstract} Let $X$ be an algebraic variety over an algebraically closed field $k$ of zero characteristic. Suppose that $G$ is a virtually polycyclic subgroup in the group $\Bir(X)$ of birational automorphisms of $X$. We show that the virtual derived length of $G$ does not exceed $2\dim(X)$. Moreover, if $X$ is a surface then this bound can be improved to $3$ and this value is optimal.
\end{abstract}

\maketitle

\section{Introduction}

We work over an algebraically closed field $k$ of zero characteristic. Let $X$ be an algebraic variety, that is, an integral and separated scheme of finite type over $k$. We denote by $\Aut(X)$ (respectively, $\Bir(X)$) the group of biregular (respectively, birational) automorphisms of $X$. 

In this note we are interested in {\em finitely generated solvable} subgroups of birational automorphism groups. Recall that a group $G$ is {\em solvable} if there exists a finite sequence of subgroups \begin{equation}\label{series1} 
G = G_0 \rhd G_1 \rhd G_2 \rhd \cdots \rhd G_{k-1} \rhd G_{k} = \{1\},
\end{equation}
such that the successive quotients $G_i/G_{i+1}$ are abelian groups. Equivalently, the {\em derived series} of $G$, defined inductively by $G^{(0)} = G$ and $G^{(i)} = [G^{(i-1)}, G^{(i-1)}]$ has to terminate at $G^{(k)} = \{1\}$. The length of the derived series is the minimal $k$ such that the group $G$ admits a series \eqref{series1} of length $k$. This number is called the {\em derived length} of $G$ and denoted by $\dl(G)$. More generally, a group is {\em virtually solvable} if it contains a solvable subgroup of finite index. For a virtually solvable group $G$ one may define $$\vdl(G) = \min\{\dl(H) \mid H \subset G \mbox{ is a solvable finite index subgroup}\}.$$

For various groups associated to geometric objects, such as automorphism and birational automorphism groups of varieties, one may consider the quantity $$\psi(K) = \sup\{\dl(G) \mid G \subset K \mbox{ is a solvable subgroup}\}.$$
In \cite[Proposition 3.14]{FP18} J.-P. Furter and P.-M. Poloni have shown that $\psi(\Aut(\mathbb{A}^2)) = 5$. For birational automorphism groups, the first such result, due to J. D\'eserti \cite[Th\'eor\`eme 1.1]{Des07}, is for arbitrary nilpotent subgroups of the Cremona group $\mathrm{Cr}_2(\mathbb{C})$.

\begin{theorem}[J. D\'eserti]\label{deserti} Let $N$ be a nilpotent subgroup of the Cremona group $\mathrm{Cr}_2(\mathbb{C}) = \Bir(\PP^2_{\mathbb{C}})$. Then either $N$ is a torsion group or the virtual derived length of $N$ is at most two.
\end{theorem}

For arbitrary solvable subgroups of $\mathrm{Cr}_2(\mathbb{C})$ the optimal upper bound on the derived length is equal to 5; this has been shown very recently in \cite[Theorem 1]{FH26}. The previously known upper bound on the derived length was 8 (see \cite[Theorem 1.10]{Ure21}). Also, for zero-entropy subgroups in automorphism groups of {\em projective} varieties (and compact Kaehler manifolds) $X$ there are known bounds on virtual invariants in terms of $\dim(X)$ and the Kodaira dimension (see \cite{DLOZ26}). The only previously known analogue of Theorem \ref{deserti} regarding varieties of arbitrary dimension was obtained by M. Abboud \cite[Theorem 1.6]{Abb23} using the $p$-adic method, developed earlier by S. Cantat and J. Xie in \cite{CX18}.

\begin{theorem}[M. Abboud]\label{abboud} Let $X$ be a quasi-projective algebraic variety over an algebraically closed field $k$ of zero characteristic. If $H$ is a finitely generated nilpotent subgroup of $\Aut(X)$, then the virtual derived length of $H$ is bounded from above by $\dim(X)$. Moreover, this bound is sharp.
\end{theorem}

It is natural to try to generalize Theorem \ref{abboud} in two directions: first, to consider subgroups in $\Bir(X)$ instead of $\Aut(X)$, and second, to pass from finitely generated nilpotent to finitely generated solvable groups, or at least to some appropriate subclass of solvable groups. An important class of finitely generated solvable groups is that of polycyclic groups. These can be characterized as finitely generated solvable subgroups with finiteness conditions for subgroups (see Proposition \ref{polycycequiv} below). If $G \subset \Bir(X)$ is a virtually polycyclic subgroup, we are interested in establishing a bound on the virtual derived length of $G$ in terms of the dimension of $X$.

The main result of the present note is the following version of Theorem \ref{abboud} for subgroups of birational automorphism groups.

\begin{theorem}\label{main}  Let $X$ be an algebraic variety over an algebraically closed field $k$ of zero characteristic. Suppose that $G \subset \Bir(X)$ is a virtually polycyclic subgroup. Then the virtual derived length of $G$ is bounded from above by $2\dim(X)$.
\end{theorem}

The idea of the proof is as follows: \begin{itemize} \item apply the structure theorem for polycyclic groups (Theorem \ref{maltsev}) in order to reduce the question to the case of finitely generated nilpotent groups; \item use the criterion (due to A. Lonjou and Ch. Urech \cite{LU21}) of regularization and pseudo-regularization in codimension $l$ for subgroups $G \subset \Bir(X)$ in terms of fixed points of actions on certain median graphs; \item use the known results on stable seminorms on nilpotent and polycyclic groups (from \cite{GS91, Con98, Con00, Gen22}) in order to prove (pseudo-)regularization for certain subgroups of the group $G$ in question, and thus reduce the question to Theorem \ref{abboud}.\end{itemize}

The notion of a stable norm, or a translation length function on a group (Definition \ref{trlength}) is a convenient tool to study various ``rigidity'' questions in geometric group theory, such as existence of fixed points for subgroups of a given group. Such functions have been applied to the study of birational automorphism groups, primarily in the surface case, see e. g. \cite{BF19} and \cite{Lam26}. The translation length functions considered here are quite special. In fact, the results we need (Theorem \ref{conner1}) can be expressed in terms of the so-called relative (FW) property for subgroups. However, we hope that a more general setting will be useful for studying more general functions of this kind, such as (symmetrizations of) logarithms of dynamical degrees (see \cite[\S 2.11]{Ye25}).

The optimality of the bound $2\dim(X)$, at least for some values of $\dim(X)$, remains an open question. However, in the case of surfaces we are able to establish the stronger (and optimal) upper bound, using the same techniques. 

\begin{theorem}\label{main2} Let $S$ be an algebraic surface over an algebraically closed field $k$ of zero characteristic. Suppose that $G \subset \Bir(S)$ is a virtually polycyclic subgroup. Then the virtual derived length of $G$ does not exceed $3$. Moreover, this bound is optimal.
\end{theorem}

\begin{remark}\label{autopt} For a virtually polycyclic subgroup $G \subset \Aut(X)$, where $X$ is a $d$-dimensional algebraic variety, the upper bound on $\vdl(G)$ is equal to $d+1$. This bound is optimal in every  dimension $d \geqslant 1$. For completeness, we provide a uniform construction of polycyclic subgroups $G_d \subset \Aut(\mathbb{A}^d)$ with $\vdl(G_d) = d+1$ in Example \ref{polyaut} below.
\end{remark}

The structure of the paper is as follows. In Section 2 we recall a few general definitions and results regarding seminorms and translation length functions on groups. Also, in this section we briefly describe the construction of certain median graphs \cite{LU21, Lon25} necessary for the regularization criterion. In Section 3 we recall the basic definitions and results on polycyclic groups and on translation length functions on them. Finally, Section 4 is devoted to the proofs of the main results (Theorems \ref{main} and \ref{main2}).

\textbf{AI disclosure.} For the current version of the paper the author used OpenAI’s ChatGPT to investigate possible proofs, examine preliminary arguments, verify computations, and locate relevant literature. He independently checked the mathematical arguments and prepared the final text, for which he accepts full responsibility.

\textbf{Acknowledgements.} The author is grateful to C. Shramov and A. Kuznetsova for many interesting and stimulating discussions. Support from the Basic Research Program of HSE University is gratefully acknowledged (HSE-BR-2025-060).

\section{Seminorms and translation length functions on groups}

\begin{definition} \label{norm} Let $G$ be a group. A {\em seminorm} on $G$ is a function $||.|| \colon G \to \mathbb{R}_{\geqslant 0}$ satisfying the following properties: \begin{enumerate} \item $||e_G|| = 0$, where $e_G \in G$ is the neutral element; \item $||g\cdot h|| \leqslant ||g|| + ||h|| \quad \forall g,h \in G$; \item $||g|| = ||g^{-1}|| \quad \forall g \in G$. \end{enumerate} A seminorm $||.|| \colon G \to \mathbb{R}_{\geqslant 0}$ is a {\em norm} if $||g|| = 0$ implies $g = e_G$. A {\em pseudometric} on a group $G$ is a function $d \colon G\times G \to \mathbb{R}_{\geqslant 0}$ such that \begin{enumerate} \item $d(g, g) = 0 \quad \forall g \in G$; \item $d(g, h) = d(h, g) \quad \forall g, h \in G$; \item $d(g, f) \leqslant d(g, h) + d(h, f) \quad \forall g, h, f \in G$. \end{enumerate}

Observe that a (semi)norm on $G$ determines a left-invariant (pseudo)metric $d \colon G\times G \to \mathbb{R}_{\geqslant 0}$ by the rule $d(g, h) = ||h^{-1}g||$.
\end{definition}

\begin{definition} \label{trlength} Let $G$ be a group endowed with a (semi)norm $||.||$. For $g \in G$ the number $$\tau(g) = \limsup_{n \to \infty}\frac{||g^n||}{n}$$ is called the {\em stable (semi)norm} of $g$. Obviously, if $g \in G$ is an element of finite order, then $\tau(g) = 0$. An element $g \in G$ of infinite order such that $\tau(g) = 0$ is called {\em distorted} with respect to the (semi)norm $||.||$.
\end{definition}

Note that a stable seminorm is not, in general, a seminorm in the sense of Definition \ref{norm}.

\begin{remark}\label{fekete} Note that the sequence of non-negative real numbers $\{||g^n||, n \in \mathbb{N}\}$ is subadditive, that is, $||g^{m+k}|| \leqslant ||g^m|| + ||g^k||$. The well-known Fekete lemma (see e. g. \cite[Lemma I.4.1]{Lam26}) implies that the upper limit in the definition of $\tau(g)$ is actually a limit: $$\tau(g) = \limsup_{n \to \infty}\frac{||g^n||}{n} = \lim_{n \to \infty}\frac{||g^n||}{n} = \inf_{n \in \mathbb{N}}\frac{||g^n||}{n}.$$
\end{remark}

The most important examples of norms on finitely generated groups are given by the word norms associated to finite generating subsets.

\begin{example}[Word norms] Let $G$ be a finitely generated group. Fix a finite generating subset $S = \{s_1, \ldots, s_n\}$ of $G$. For every element $g \in G$ define $||g||_S$ to be the length of the shortest word on the alphabet $S^{\pm 1}$ representing the element $g$. The resulting function is clearly a norm on $G$, and the corresponding left-invariant metric $d_S$ is the {\em word metric} on the Cayley graph of $G$ with respect to $S$.
\end{example}

The other fundamental example comes from metric geometry. The so-called translation length functions are related to fixed-point properties for groups of isometries.

\begin{example}[Translation lengths for actions by isometries] Let $(M, d)$ be a metric space. Suppose that a group $G$ acts on $M$ by isometries of the metric $d$. For $g \in G$ the {\em displacement function} $d_g \colon M \to \mathbb{R}_{\geqslant 0}$ is defined by $d_g(x) = d(x, gx)$. The {\em translation length} of $g$ is the number $$|g| = \inf\{d_g(x) \mid x \in M\}.$$ Also, for every point $x \in M$ one can view the dispacement as a seminorm $||g||_x = d(x, gx)$ on the group $G$. The {\em minimal set} $\mathrm{Min}(g)$ of $g$ is the set of points where $d_g$ attains its infimum; if $\mathrm{Min}(g) \neq \emptyset$ the isometry $g$ is called {\em semisimple}. If $g \in G$ is a semisimple isometry of a CAT(0) metric space, then one has $$|g| = \tau(g) = \limsup_{n \to \infty}\frac{d(x, g^nx)}{n},$$ where the upper limit is independent of $x \in M$ (see \cite[Chapter II.6]{BH99}).
\end{example}

In what follows we use the names ``stable (semi-)norm'' and ``translation length function'' interchangeably for functions $\tau \colon G \to \mathbb{R}_{\geqslant 0}$ as in Definition \ref{trlength}. In the next proposition we list some properties of these functions that follow directly from the definitions. For the proofs we refer to \cite[Appendix]{GS91}.

\begin{proposition} \label{properties} Let $G$ be a group endowed with a seminorm $||.||$, and let $g, h \in G$. Then the following statements are true. \begin{enumerate} \item $0 \leqslant \tau(g) \leqslant \frac{1}{n}||g^n||$ for all $n \in \mathbb{N}$; \item $\tau(h) = \tau(ghg^{-1})$; \item $\tau(g^n) = |n|\tau(g)$ for all $n \in \mathbb{Z}$; \item If $gh = hg$ then $\tau(gh) \leqslant \tau(g) + \tau(h)$.\end{enumerate}
\end{proposition}

\begin{remark}\label{length} Some authors reserve the name ``length functions'' for a more general class of functions $l \colon G \to \mathbb{R}_{\geqslant 0}$ that satisfy properties (2), (3) and (4) of Proposition \ref{properties}. These functions are not necessarily obtained as translation length functions associated to a seminorm on $G$. We refer to \cite{Ye25} for numerous examples of such functions on groups.
\end{remark}

The following definitions are due to G. Conner \cite[Definition 2.6]{Con98}.

\begin{definition}\label{discrete} A group $G$ is called \begin{itemize} 
\item {\em translation proper}, if there exists a norm $||.||$ on $G$ such that for every infinite order element $g \in G$ one has $\tau(g) > 0$;
\item {\em translation discrete}, if there exists a norm $||.||$ on $G$ such that the set of values $\{\tau(g)\}$ where $g$ is a non-torsion element of $G$ is bounded away from $0$.
\end{itemize}
\end{definition}

Following \cite[Definition 3.3]{Con98}, we also define the seminorm induced on the quotient group $G/H$ by a seminorm $||.||$ on $G$. 

\begin{definition}\label{induced} Let $G$ be a group endowed with a seminorm $||.||$ and let $H \subset G$ be a normal subgroup. The {\em induced seminorm} on the quotient group $G/H$ is defined by $$||gH||_{G/H} = \inf_{x \in gH}||x||.$$ It is clear from this definition that $$\tau_{G}(g) \geqslant \tau_{G/H}(\pi(g))$$ for every $g \in G$, where $\pi \colon G \to G/H$ is the quotient homomorphism. 
\end{definition}

\subsection{Translation lengths on $\Bir(X)$ from actions on median graphs}

Let $X$ be an algebraic variety. In this subsection we recall the notions of regularization and pseudo-regularization in codimension $l$ (where $0 \leqslant l \leqslant \dim(X)$) for elements and subgroups of $\Bir(X)$. Following \cite{LU21, Lon25}, we state the criteria for regularization and pseudo-regularization in codimension $l$ in terms of fixed point properties for actions on median graphs (or, equivalently, $\mathrm{CAT}(0)$ cube complexes) $\mathcal{C}^l(X)$.

Recall that if $f \colon X \dasharrow Y$ is a birational map of algebraic varieties, the {\em exceptional set} of $f$ is the subset of points $x \in X$ such that $f$ is not a local isomorphism in any neighborhood of $x$.

\begin{definition}\label{psaut} A birational map $f \colon X \dasharrow Y$ between algebraic varieties is a {\em pseudo-isomorphism in codimension $l$} if the exceptional sets of $f$ and $f^{-1}$ are of codimension $> l$ in $X$ and $Y$, respectively. For an algebraic variety $X$ the group of pseudo-automorphisms in codimension $l$ is denoted by $\Psaut^l(X)$. One has a sequence of subgroups $$\Aut(X) = \Psaut^d(X) \subset \Psaut^{d-1}(X) \subset \cdots \subset \Psaut^1(X) \subset \Psaut^0(X) = \Bir(X),$$ where $d$ is the dimension of $X$. Usually, the group $\Psaut^1(X)$ is denoted simply by $\Psaut(X)$ and called the group of pseudo-automorphisms of $X$. 
\end{definition}

\begin{definition}\label{psreg} A birational automorphism $f \in \Bir(X)$ (respectively, a subgroup $G \subset \Bir(X)$) is called {\em pseudo-regularizable in codimension $l$} if there exists an algebraic variety $Y$ and a birational map $\varphi \colon Y \dasharrow X$ such that $\varphi^{-1} \circ f \circ \varphi \in \Psaut^{l}(Y)$ (respectively, $\varphi^{-1} \circ g \circ \varphi \in \Psaut^{l}(Y)$ for every $g \in G$). If the conjugating birational map $\varphi \colon Y \dasharrow X$ can be chosen to be an isomorphism in codimension $m$, then we say that $f$ (or $G$) is {\em pseudo-regularizable in codimension $l$ by a pseudo-isomorphism in codimension $m$}.
\end{definition}

The next lemma \cite[Lemma 2.10]{LU21} says that a regularizable subgroup can always be regularized on a quasi-projective variety.

\begin{lemma}\label{quasipr} Let $X$ be an algebraic variety and let $G \subset \Aut(X)$ be an arbitrary subgroup. Then there exists a quasi-projective variety $Y$ and a birational map $\varphi \colon Y \dasharrow X$ such that $\varphi^{-1} \circ G \circ \varphi \subset \Aut(Y)$.
\end{lemma}

To study pseudo-regularization (in codimension $l+1$) of elements $f \in \Psaut^l(X)$ the authors of \cite{LU21} introduced the following asymptotic invariant.

\begin{definition}\label{dynnum} Let $f \colon X \dasharrow X$ be a pseudo-automorphism in codimension $l$. Let $\Exc^{l+1}(f)$ be the set of irreducible components of pure codimension $l+1$ of the exceptional locus of $f$. The function $$\nu^{l+1}(f) = \limsup_{n \to \infty}\frac{|\Exc^{l+1}(f^{\circ n})|}{n}$$ is called the {\em dynamical number of the (l+1)-exceptional locus}.
\end{definition}

The criterion for pseudo-regularization can be stated as follows \cite[Propositions 4.16(4) and 4.17]{LU21}.

\begin{theorem}\label{lureg} An element $f \in \Psaut^{l}(X)$, where $0 \leqslant l \leqslant \dim(X) - 1$, is pseudo-regularizable in codimension $l+1$ by a pseudo-isomorphism in codimension $l$ if and only if $\nu^{l+1}(f) = 0$. A subgroup $G \subset \Psaut^{l}(X)$ is pseudo-regularizable in codimension $l+1$ by a pseudo-isomorphism in codimension $l$ if and only if the subset $\{|\mathrm{Exc}^{l+1}(g)| \mid g \in G\} \subset \mathbb{Z}_{\geqslant 0}$ is bounded.
\end{theorem}

A. Lonjou and Ch. Urech proved this criterion by constructing natural actions of the groups $\Psaut^l(X)$ by isometries on certain median graphs. A {\em median graph} is a connected graph $\Gamma$ such that for any triple of vertices $u, v, w \in \Gamma$ there exists a unique vertex $m \in \Gamma$ (called the {\em median point}) such that the equalities $d(u, v) = d(u, m) + d(m, v)$, $d(u, w) = d(u, m) + d(m, w)$ and $d(w, v) = d(w, m) + d(m, v)$ hold. Here $d$ stands for the combinatorial metric on $\Gamma$. There exists a rich theory of groups acting on median graphs, vastly generalizing the more familiar theory of groups acting on simplicial trees. We refer to \cite{GenBook, Gen22} and references therein for more details.

\begin{definition}\label{medgraphs} Fix the number $l \in \{0, \ldots, \dim(X)-1\}$ and define the graph $\mathcal{C}^l(X)$ as follows. The vertices of $\mathcal{C}^l(X)$ are equivalence classes of pairs $(A, \varphi)$ where $A$ is an algebraic variety and $\varphi \colon A \dasharrow X$ is a birational map which is an isomorphism in codimension $l$. Two pairs $(A, \varphi)$ and $(B, \psi)$ are equivalent if the induced birational map $\psi^{-1}\circ \varphi \colon A \dasharrow B$ is an isomorphism in codimension $l+1$. The class of the pair $(A, \varphi)$ is denoted by $[(A, \varphi)]$. Two vertices $v_1$ and $v_2$ are connected by an edge (oriented from $v_1$ to $v_2$) if there exists a pair $(A, \varphi)$ representing $v_1$ and an irreducible closed subvariety $D \subset A$ of codimension $l+1$ such that the vertex $v_2$ is represented by $(A\setminus D, \varphi|_{A\setminus D})$.
\end{definition}

Note that the group $\Psaut^l(X)$ acts on the graph $\mathcal{C}^l(X)$ by the following rule: the element $f \in \Psaut^l(X)$ sends the vertex $[(A, \varphi)]$ to $[(A, f \circ \varphi)]$. In \cite{LU21} it is shown that this action of $\Psaut^l(X)$ on $\mathcal{C}^l(X)$ is well-defined, faithful, and preserves the metric and orientation. The induced isometries of the median graphs can be classified. Recall that a {\em cube} in a median graph is a subgraph isomorphic to the 1-skeleton of a cube.

\begin{definition}\label{isometries} Let $\Gamma$ be a median graph and let $g \colon \Gamma \to \Gamma$ be an isometry. Then $g$ is called \begin{enumerate} \item {\em Combinatorially elliptic} or just {\em elliptic} if $g$ fixes a vertex; \item {\em Periodic} if $g$ preserves some cube in $\Gamma$; \item {\em Combinatorially loxodromic} or just {\em loxodromic} if $g$ preserves a doubly infinite geodesic and acts on it by a non-trivial translation. \end{enumerate} 
A {\em cubical orientation} on a median graph $\Gamma$ is a choice of an orientation on each edge of $\Gamma$ such that every pair of opposite edges of a square in $\Gamma$ have the same orientation. In this paper all orientations on median graphs are cubical.
\end{definition}

The key properties of the actions of $\Psaut^l(X)$ on $\mathcal{C}^l(X)$ are summarized in the theorem below (see \cite[Section 4]{LU21} and \cite[Section 5]{Lon25} for the proofs). In particular, it shows that the function $\nu^{l+1}$ is a stable seminorm on the group $\Psaut^l(X)$.

\begin{theorem} \label{lonjouurech} Let $X$ be an algebraic variety of dimension $d$ over an algebraically closed field $k$ and let $l \in \{0, \ldots, d-1\}$. Then the following statements are true. \begin{enumerate} \item The graph $\mathcal{C}^l(X)$ from Definition \ref{medgraphs} is a median graph; \item The combinatorial distance between vertices $v_1 = [(A, \varphi)]$ and $v_2 = [(B, \psi)]$ in $\mathcal{C}^l(X)$ is equal to $$d(v_1, v_2) = |\mathrm{Exc}^{l+1}(\psi^{-1}\circ\varphi)| + |\mathrm{Exc}^{l+1}(\varphi^{-1}\circ\psi)|;$$ \item An isometry of $\mathcal{C}^l(X)$ induced by an element $f \in \Psaut^l(X)$ is either combinatorially elliptic or combinatorially loxodromic; \item the translation length function for the action of $\Psaut^l(X)$ on $\mathcal{C}^l(X)$ is equal to $$\tau(f) = \nu^{l+1}(f) + \nu^{l+1}(f^{-1}) = 2\nu^{l+1}(f),$$ and moreover the function $\nu^{l+1}$ is integer-valued. \end{enumerate}
\end{theorem}

In view of Theorem \ref{lonjouurech} the pseudo-regularization criterion (Theorem \ref{lureg}) can be restated in the following geometric way.

\begin{theorem}\label{lureg2} Let $X$ be an algebraic variety and let $0 \leqslant l \leqslant \dim(X) - 1$. A subgroup $G \subset \Psaut^{l}(X)$ is pseudo-regularizable in codimension $l+1$ by a pseudo-isomorphism in codimension $l$ if and only if the natural action of $G$ on $\mathcal{C}^l(X)$ fixes a vertex. 
\end{theorem}

\begin{remark}\label{blowup} The median graphs $\mathcal{C}^l(X)$ described above generalize to higher-dimensional varieties the construction of the {\em blow-up complex} $\mathcal{C}(S)$ of a smooth projective surface $S$ (see \cite[Section 3]{LU21}). It is a median graph, and the translation length function of the action of $\Bir(S)$ on $C(S)$ is (twice) the {\em dynamical number of base points} $$\nu(f) = \limsup_{n\to \infty}\frac{b(f^n)}{n},$$ where $b(f^n)$ denotes the number of base points of $f^n$ (including infinitely near ones). In \cite[Theorem 3.9 and Proposition 3.11]{LU21} it was shown that a subgroup $G \subset \Bir(S)$ is regularizable on a {\em projective} surface $S'$ birational to $S$ if and only if the action of $G$ on the blow-up complex fixes a vertex.
\end{remark}

\section{Translation lengths on polycyclic groups}

\subsection{Generalities on (virtually) polycyclic groups}

\begin{definition}\label{polycyclic} A group $G$ is called {\em polycyclic} if there exists a finite sequence of subgroups \begin{equation} \label{series} G = G_0 \rhd G_1 \rhd \cdots \rhd G_n = \{1\}\end{equation} such that for every $i$ the subgroup $G_{i+1}$ is normal in $G_{i}$ and the quotient group $G_{i}/G_{i+1}$ is cyclic. A group is {\em virtually polycyclic} if it contains a polycyclic subgroup of finite index.
\end{definition}

It is immediate from the definition that every polycyclic group is solvable. Conversely, there are several ways to characterize polycyclic groups among solvable ones (see \cite[Proposition 4]{Seg83} for the proof).

\begin{proposition}\label{polycycequiv} A group $G$ is polycyclic if and only if it satisfies one of the following equivalent conditions: \begin{enumerate} \item $G$ is solvable and has the max condition on subgroups: every nonempty family of subgroups in $G$ has a maximal member; \item $G$ is solvable and noetherian, that is, every subgroup of $G$, including $G$ itself, is finitely generated. \end{enumerate}
\end{proposition}

For the next result see e. g. \cite[Propositions 13.80 and 13.85]{DK18}.

\begin{proposition}\label{hirsch} Let $G$ be a polycyclic group. Then there exists a subgroup of finite index $\overline{G} \subset G$ which is {\em poly-infinite-cyclic}, that is, admits a series \eqref{series} such that $\overline{G}_i/\overline{G}_{i+1} \cong \mathbb{Z}$ for all $i$. The number of infinite cyclic quotients in any series \eqref{series} is an invariant of the group $G$, called the {\em Hirsch number} of $G$ and denoted by $h(G)$. If $G$ is a virtually polycyclic group, one may define $h(G) = h(H)$, where $H \subset G$ is any polycyclic subgroup of finite index.
\end{proposition}

Here are a few important properties of polycyclic groups that we will need (see \cite[Propositions 13.73 and 13.75]{DK18} and \cite[Exercise 8]{Seg83}).

\begin{proposition}\label{polycprop} \begin{enumerate} \item The class of polycyclic groups is closed under taking subgroups, quotient groups and extensions; \item If $G$ is a polycyclic group and $N \lhd G$ is a normal subgroup then $h(G) = h(N) + h(G/N)$; \item A nilpotent group is polycyclic if and only if it is finitely generated.\end{enumerate}
\end{proposition}

The next theorem is a celebrated result of A. I. Mal'cev \cite{Mal} regarding the structure of polycyclic groups.

\begin{theorem}[A. I. Mal'cev]\label{maltsev} Let $G$ be a polycyclic group. Then there exists a finite index subgroup $\overline{G} \subset G$ that is an extension $$1 \to N \to \overline{G} \to A \to 1,$$ where $N$ is a finitely generated nilpotent group, and $A$ is a finitely generated abelian group.
\end{theorem}

A solvable group is called {\em linear} if it admits an embedding to $\mathrm{GL}_n(F)$ for some $n \in \mathbb{N}$ and some field $F$.

\begin{theorem}[A. I. Mal'cev]\label{maltsevlin} Let $G$ be a linear solvable group. Then $G$ is virtually an extension of an abelian group by a nilpotent group.
\end{theorem}

In fact, according to results of A. I. Mal'cev and L. Auslander--R. Swan (see e. g. \cite{Seg83}) polycyclic groups are precisely those solvable groups that are linear over $\mathbb{Z}$.

It is well-known that the class of solvable groups is closed under extensions, and the derived length is subadditive (see e. g. \cite[Proposition 13.91]{DK18}). In \cite[Theorem 5]{KKMM09} it was shown that if a group $G$ has a virtually solvable normal subgroup $H$ such that $G/H$ is also virtually solvable then $G$ itself is virtually solvable. From this result and Proposition \ref{polycycequiv} it follows easily that the extension of two virtually polycyclic groups is also virtually polycyclic. First, it is clear that the virtual derived length is invariant under passing to a subgroup of a finite index. The next lemma implies that the virtual derived length is subadditive for the class of virtually polycyclic groups.

\begin{lemma}\label{virt} Suppose that $H$ is a finitely generated virtually solvable group such that $\mathrm{vdl}(H) \leqslant n$. Consider an extension $$1 \to H \to G \to A \to 1,$$ where $A$ is a finitely generated virtually abelian group. Then the group $G$ is virtually solvable, and $\mathrm{vdl}(G) \leqslant n+1$.
\end{lemma}

\begin{proof} Passing to a subgroup of finite index, we may assume $A$ to be abelian. Let $\overline{H} \subset H$ be a solvable subgroup of finite index and derived length equal to $\vdl(H)$. Since $H$ is finitely generated, there are only finitely many subgroups in $H$ of the same index as $\overline{H}$. Hence, replacing $\overline{H}$ by the intersection of these subgroups, we may assume $\overline{H}$ to be characteristic in $H$ and thus normal in $G$. The quotient group $Q = G/\overline{H}$ is an extension $$1 \to F \to Q \to A \to 1,$$ where $F = H/\overline{H}$ is finite and $A$ is finitely generated abelian by assumption. We will show that the group $Q$ is virtually abelian (cf. \cite[Lemma 1.A.3]{Seg83}). Let $\{g_1, \ldots, g_k\}$ be a finite generating set of $Q$. Since $A$ is abelian, for any elements $g, h \in Q$ one has $$h^{-1}ghF = gF,$$ which implies that the conjugacy class of $g$ is contained in a finite coset $gF$. On the other hand, the cardinality of the conjugacy class of $g$ is equal to the index of the centralizer $C_g(Q)$ in $Q$. Thus, for every element $g \in Q$ one has $$[Q\colon C_g(Q)] \leqslant |F|.$$ Note that an element $g \in Q$ lies in the center of $Q$ if and only if it commutes with every generator $g_1, \ldots, g_k$. Therefore, one has $$Z(Q) = \bigcap_{1\leqslant i \leqslant k}C_{g_i}(Q),$$ which implies that the index of $Z(Q)$ in $Q$ does not exceed $$\prod_{1\leqslant i\leqslant k}[Q\colon C_{g_i}(Q)] \leqslant |F|^k < \infty.$$ Hence, the center $Z(Q)$ is a finitely generated abelian subgroup of finite index in $Q$.

Now let $\overline{G}$ be the preimage of $Z(Q)$ in $G$. The subgroup $\overline{G} \subset G$ is of finite index and is an extension of the abelian group $Z(Q)$ by a solvable group $\overline{H}$. Therefore, $\overline{G}$ is solvable of derived length not exceeding $\dl(\overline{H})+1 = \vdl(H)+1$ (see e.g. \cite[Proposition 13.91]{DK18}). This finishes the proof of the Lemma. 
\end{proof}

\begin{remark} If the abelian group $A$ in Lemma \ref{virt} is not assumed to be finitely generated, the bound $n+1$ does not hold. In fact, in Section 5 of \cite{KKMM09} it is shown that if a virtually solvable group $G$ is an extension of two virtually solvable groups $H$ and $G/H$ then $$\mathrm{vdl}(G) \leqslant \mathrm{vdl}(H) + \mathrm{vdl}(G/H) + 1,$$ and that this bound is optimal in general.
\end{remark}

\subsection{Translation length functions on polycyclic groups}

In this subsection we state a few definitions and general results related to translation length functions on nilpotent and polycyclic groups, mostly following G. Conner's works \cite{Con98, Con00}. Also, we state here a general structure theorem for actions of polycyclic groups on median graphs (or $\mathrm{CAT}(0)$ cube complexes), due to A. Genevois \cite{Gen22}.

First, we state a purely group-theoretic result \cite[Theorem 3.5]{Con98} regarding arbitrary stable seminorms on nilpotent groups.

\begin{theorem}\label{conner1} Let $N$ be a nilpotent group endowed with a seminorm $||.||$ and let $\tau$ be the corresponding stable seminorm. Then the subset $$I(N) = I(N, \tau) = \{g \in N \mid \tau(g) = 0\}$$ is a normal subgroup of $N$. Moreover, one has $$I(N, \tau) \supset \sqrt{N'} = \{g \in N \mid g^k \in N' \mbox{ for some } k \in \mathbb{N}\}.$$
\end{theorem}

Observe that in the assumptions of Theorem \ref{conner1} the quotient group $N/I(N)$ is translation proper with respect to the induced seminorm. So from Theorem \ref{conner1} it follows immediately that a finitely generated nilpotent group is translation proper if and only if it is virtually abelian. In \cite{Con00} G. Conner obtained generalizations of this statement for certain classes of finitely generated solvable groups.

In what follows we consider stable seminorms induced by actions of groups on median graphs defined in Section 2. We state an important result regarding actions of groups on median graphs by isometries with bounded orbits, see \cite[Theorem 4.1.1 and Corollary 4.1.10]{GenBook} for proofs.

\begin{lemma}\label{fixedptlemma} Let $G$ be a group acting by isometries on a median graph $\Gamma$. Suppose that there exists a vertex $v \in \Gamma$ such that the orbit $G\cdot v$ is bounded. Then the action of $G$ on $\Gamma$ stabilizes a cube. Moreover, if $G$ acts by isometries preserving a cubical orientation, then the action of $G$ fixes a vertex of $\Gamma$.
\end{lemma}

Now we state a theorem (see \cite[Theorem 11.2.6]{GenBook} and \cite[Theorem 5.11]{Gen22}) describing general actions of polycyclic groups on median graphs. See \cite[Definition 3.1]{Gen22} for the definitions of {\em median subalgebras} and {\em median flats} in a median graph. In the case we consider, all isometries will be orientation-preserving; therefore, periodic isometries always fix a vertex (see \cite[Corollary 4.1.10]{GenBook}).

\begin{theorem}\label{genevois} Let $G$ be a polycyclic group acting on a median graph (or a $\mathrm{CAT}(0)$ cube complex) $\Gamma$. Then there is a finite index subgroup $H \subset G$ such that \begin{itemize} \item The subset $$E = \{g \in H \mid \mbox{$g$ acts on $\Gamma$ by a periodic isometry}\}$$ is a normal subgroup of $H$, and the quotient $H/E$ is a free abelian group; \item the subgroup $E$ stabilizes a median subalgebra which is a median flat or a single vertex.\end{itemize}
\end{theorem}

Suppose that a group $G$ acts on a median graph $\Gamma$ such that every element of $G$ fixes a vertex of $\Gamma$. In general, it is not true that there exists a vertex of $\Gamma$ fixed by the entire group $G$. Hence, we will need the following simple proposition regarding such actions of virtually polycyclic groups.

\begin{proposition}\label{fixedpt} Let $G$ be a virtually polycyclic group. Suppose that $G$ acts on a median graph $\Gamma$ in such a way that all elements of $G$ act by elliptic and orientation-preserving isometries. Then the action of $G$ fixes a vertex of $\Gamma$.
\end{proposition}

\begin{proof} First, it suffices to prove the statement for a finite index subgroup $H \subset G$. Indeed, suppose that $H$ fixes a vertex $v_0 \in \Gamma$. Then for each $g \in G$ the vertex $g\cdot v_0$ depends only on the coset $gH \subset G$. Since $H \subset G$ is of finite index, the orbit $Gv_0$ of the vertex $v_0$ under $G$ is finite. Therefore, the action of $G$ on $\Gamma$ has a bounded orbit and preserves orientation, hence, by Lemma \ref{fixedptlemma}, $G$ fixes a vertex of $\Gamma$.

Hence, by Proposition \ref{hirsch} we may assume $G$ to be poly-infinite-cyclic. The proof is by induction on the Hirsch length of $G$. The induction base is clear; suppose that the Hirsch length of $G$ is equal to $k$. Let $G_1$ be a normal subgroup of $G$ such that $G/G_1 \cong \mathbb{Z}$ and $h(G_1) = k-1$. By the induction hypothesis, the subgroup $G_1$ has a fixed vertex $v_0 \in \Gamma$. Let $g \in G$ be a preimage of a generator of $G/G_1$; then any element of $G$ can be written in the form $g^n\cdot h$ for some $n \in \mathbb{Z}$ and $h\in G_1$. Therefore, one has $$Gv_0 = \{g^n\cdot v_0 \mid n \in \mathbb{Z}\}.$$ Since $g$ is elliptic by assumption, there exists a vertex $v_1 \in \Gamma$ such that $g\cdot v_1 = v_1$. Then for any $n \in \mathbb{Z}$ one has $$d(v_0, g^n\cdot v_0) \leqslant d(v_0, v_1) + d(v_1, g^n\cdot v_0) = d(v_0, v_1) + d(g^nv_1, g^nv_0) = 2d(v_0, v_1),$$ so the orbit $Gv_0$ is bounded.

Therefore, the action of $G$ on $\Gamma$ has a bounded orbit and preserves a cubical orientation, so by Lemma \ref{fixedptlemma} $G$ fixes a vertex of $\Gamma$, as desired. 
\end{proof}

\section{Proofs of the main results}

\begin{proof}[Proof of Theorem \ref{main}] Let $d \geqslant 1$ be the dimension of $X$. Passing to a subgroup of finite index, one may assume $G$ to be polycyclic. Let us consider the action of $G$ on the median graph $\mathcal{C}^0(X)$. Then by Theorem \ref{genevois}, up to replacing $G$ by a subgroup of finite index, there exists a subgroup $E \lhd G$ such that $G/E$ is free abelian and all elements of $E$ are elliptic. Since the action of $G$ on $\mathcal{C}^0(X)$ preserves the cubical orientation, by Proposition \ref{fixedpt} there exists a vertex of $\mathcal{C}^0(X)$ preserved by $E$. Therefore, by Theorem \ref{lureg2} the subgroup $E$ is pseudo-regularizable in codimension 1.

On the other hand, by Theorem \ref{maltsev}, up to replacing $G$ by a finite index subgroup, the group $G$ is an extension $$1 \to N \to G \to G/N \to 1,$$ where $N$ is a finitely generated nilpotent group and $G/N$ is a finitely generated abelian group. Replacing $E$ with its intersection with the finite index subgroup, consider the subgroup $H = E \cap N$; it is finitely generated nilpotent and pseudo-regularizable in codimension 1. Moreover, the quotient group $A = G/H$ embeds into $(G/N)\times(G/E)$, hence, it is abelian. The case $d=1$ follows from Theorem \ref{abboud}; below, we assume that $d \geqslant 2$.

We know that the action of $H$ on the median graph $\mathcal{C}^0(X)$ has a fixed vertex $v_1 = [(X_1, \varphi_1)]$. So, by Theorem \ref{lureg2} there exists a birational map $$\varphi \colon X_1 \dasharrow X$$ such that $\varphi^{-1}\circ H\circ\varphi \subset \Psaut^1(X_1)$. Therefore, one can consider the action of the subgroup $H \subset \Psaut^1(X_1)$ on the median graph $\mathcal{C}^1(X_1)$. This action has the translation length function $2\nu_{X_1}^2$ by Theorem \ref{lonjouurech}. Let us denote $$H_2 = I(H, \nu_{X_1}^2) = \{g \in H \mid \nu_{X_1}^2(g) = 0\},$$ which is a normal subgroup of $H$ containing $H' = H^{(1)}$, according to Theorem \ref{conner1}. By Proposition \ref{fixedpt}, the action of $H_2$ on $\mathcal{C}^1(X_1)$ has a fixed vertex. Therefore, by Theorem \ref{lureg} the subgroup $H_2$ is pseudo-regularizable in codimension 2. Again, let us denote $$H_3 = I(H_2, \nu_{X_2}^3) = \{g \in H_2 \mid \nu_{X_2}^3(g) = 0\}.$$ Repeating the above argument, we obtain a subnormal series $$H = H_1 \rhd H_2 \rhd H_3 \rhd \cdots \rhd H_d,$$ such that: \begin{itemize} \item For all $k \in \{1, \ldots, d\}$ the subgroup $H_k$ is pseudo-regularizable in codimension $k$ on some algebraic variety $X_k$ birational to $X$; \item For for all $k \in \{1, \ldots, d-1\}$ the normal subgroup $H_{k+1} = I(H_{k}, \nu_{X_k}^{k+1}) \lhd H_k$ contains the commutator subgroup $(H_k)'$ by Theorem \ref{conner1}, so the quotient $H_k/H_{k+1}$ is abelian and one has $$H_{k+1} \supset (H_k)' \supset H^{(k)}.$$
\end{itemize}

In particular, the subgroup $H_d \subset H$ is regularizable on an algebraic variety $X_d$. By Lemma \ref{quasipr} we may assume $X_d$ to be quasi-projective. So by Abboud's Theorem \ref{abboud}, the virtual derived length of $H_d$ is at most $d = \dim(X_d) = \dim(X)$. From repeated application of Lemma \ref{virt} it follows that $\mathrm{vdl}(H)$ is at most $d + (d-1) = 2d-1$. Therefore, one has $$\vdl(G) \leqslant \vdl(H) + \vdl(A) \leqslant (2d-1) +1 = 2\dim(X),$$ which completes the proof of Theorem \ref{main}. \end{proof}

\begin{remark}\label{optimal} It is not clear whether the bound $\vdl(G) \leqslant 2\dim(X)$ can be optimal for some values of $d = \dim(X) \geqslant 3$. It would be instructive to find the optimal bound for the Cremona group of rank 3.
\end{remark}

\begin{proof}[Proof of Theorem \ref{main2}] We may assume $G$ to be polycyclic and $S$ to be smooth and projective. Let us consider the action of $G \subset \Bir(S)$ by isometries on the blow-up complex (see Remark \ref{blowup}). From Theorem \ref{genevois} it follows that after replacing $G$ by a finite index subgroup $G$ is an extension $$1 \to H \to G \to A \to 1,$$ where the normal subgroup $H = I(G, \nu) \lhd G$ acts by elliptic isometries (hence, by Proposition \ref{fixedpt} fixes a vertex), and the quotient group $A$ is free abelian of finite rank. 

On the other hand, applying Theorem \ref{maltsev} to the group $G$, one obtains a finite index subgroup $G_1 \subset G$ and the sequence $$1 \to N \to G_1 \to A_1 \to 1,$$ where $N$ is finitely generated nilpotent and $A_1$ is finitely generated abelian. Now one has the product homomorphism $G_1 \to A\times A_1$ obtained from the above homomorphisms $G_1 \to A$ and $G_1 \to A_1$. Its image is abelian and its  kernel is equal to $H\cap N$ which is a finitely generated nilpotent and projectively regularizable subgroup of $\Bir(S)$. Hence, by Lemma \ref{virt} and Theorem \ref{abboud} one has $$\vdl(G) = \vdl(G_1) \leqslant \vdl(H\cap N) + 1 \leqslant 2 + 1 = 3,$$ as it was to be shown. By Remark \ref{autopt}, this bound is optimal.
\end{proof}

\begin{remark}\label{cremona2} In the case $X = \PP^2$ it follows from Theorems \ref{deserti} and \ref{maltsevlin} together with \cite[Theorem 5]{KKMM09} that the virtual derived length of a linear solvable subgroup $G \subset \mathrm{Cr}_2(\mathbb{C})$ does not exceed 4. 
\end{remark}

Finally, we provide for every $d\geqslant 1$ examples of polycyclic subgroups $G_d \subset \Aut(\mathbb{A}^d)$ with $\vdl(G_d) = d+1$.

\begin{example}\label{polyaut} We consider $\mathbb{A}^d = \mathrm{Spec}(k[x_1,\ldots, x_d])$. Assign to the variable $x_i$ the weight $w_i = 2^{d-i}$. Let $U$ be the subgroup of triangular polynomial automorphisms $$x_i \mapsto x_i + p_i(x_{i+1},\ldots, x_d), \quad \deg_{w}(p_i) < w_i.$$ Then $U$ is a finite-dimensional unipotent algebraic group defined over $\mathbb{Z}$. Its faithful unitriangular representation is obtained from the action on the space of polynomials of weighted degree at most $w_1$. The derived length of $U$ is equal to $d$. To see this, note that its Lie algebra $\mathfrak{u}$ contains $$D_{i, 0} = \left(\prod_{j=i+1}^d x_j\right)\frac{\partial}{\partial x_i},$$ because of the relation $w_i - 1 = \sum_{j=i+1}^d w_j$. By induction, using the relation $[D_{d-r,r}, D_{i,r}] = D_{i,r+1}$, one obtains $$D_{i, r} = \left(\prod_{j=i+1}^{d-r}x_j\right)\frac{\partial}{\partial x_i} \in \mathfrak{u}^{(r)} \quad i \leqslant d-r.$$ Thus, one has $\partial/\partial x_1 \in \mathfrak{u}^{(d-1)}$ and $\mathfrak{u}^{(d)} = 0$ by unitriangularity.

The diagonal torus $T = (\mathbb{G}_m)^d$ normalizes $U$ and one has $$[T, U] = U.$$ Therefore, the group $U \rtimes T$ has derived length $d+1$. Let us denote $$\alpha = \frac{3 + \sqrt{5}}{2}, \quad \mathcal{O} = \mathbb{Z}[\alpha].$$ Consider the subgroups $$D = \{\mathrm{diag}(\alpha^{m_1},\ldots, \alpha^{m_d})\mid m_i \in \mathbb{Z}\}, \quad G_d = U(\mathcal{O}) \rtimes D.$$ Then $U(\mathcal{O})$ embeds into the finitely generated nilpotent group $\mathrm{UT}_M(\mathcal{O})$, so $U(\mathcal{O})$ is finitely generated nilpotent and $G_d$ is polycyclic. Moreover, every finite index subgroup of $G_d$ is Zariski-dense in the connected algebraic group $U\rtimes T$. Therefore, one has $$\vdl(G_d) = \dl(U\rtimes T) = d+1.$$
\end{example}

\medskip

\flushleft{\address{National Research University Higher School of Economics,\\
International Laboratory of Mirror Symmetry and Automorphic Forms, Faculty of Mathematics, \\
Usacheva St. 6, Moscow 119048, Russia; \\
also\\
Steklov Mathematical Institute of Russian Academy of Sciences, Moscow, Russia \\
8 Gubkina St., Moscow 119991, Russia \\}}
\email{\texttt{agolota@hse.ru}}

\end{document}